\documentclass[10pt]{amsart}
\usepackage{amsmath,amsthm,amscd}
\usepackage{amssymb}
\usepackage{url}

\theoremstyle{plain}
    \newtheorem{theorem}{Theorem}[section]
    
    \newtheorem{corollary}[theorem]{Corollary}
    
    \newtheorem{lemma}[theorem]{Lemma}

\theoremstyle{definition}
    \newtheorem{definition}[theorem]{Definition}
    
    \newtheorem{remark}[theorem]{Remark}

\newcommand{\PP}{\mathbb{P}}
\newcommand{\RR}{\mathbb{R}}
\renewcommand{\AA}{\mathbb{A}}

\newcommand{\GG}{\mathbb{G}}

\newcommand{\Ocal}{\mathcal{O}}
\newcommand{\Aut}{\operatorname{Aut}}
\newcommand{\Crit}{\operatorname{Crit}}

\newcommand{\SL}{\operatorname{SL}}
\newcommand{\Rat}{\operatorname{Rat}}
\newcommand{\bfi}{{\boldsymbol{i}}}
\newcommand{\bfX}{{\boldsymbol{X}}}
\newcommand{\bfP}{{\boldsymbol{P}}}
\newcommand{\xmark}{{\boldsymbol\times}}
\newcommand{\Case}[2]{\par\noindent\framebox{\textbf{Case #1}: $\boldsymbol{#2}$}\par\noindent\ignorespaces}
\newcommand{\Pbar}{{\bar{P}}}
\newcommand{\fbar}{{\bar{f}}}
\newcommand{\gbar}{{\bar{g}}}
\newcommand{\bfx}{{\boldsymbol{x}}}
\newcommand{\stab}{{\textup{stab}}}
\newcommand{\semistab}{{\textup{ss}}}
\newcommand{\Moduli}{\mathcal{M}}

\begin{document}

\title{GIT Stability of H\'{e}non maps}

\author[C. G. Lee]{Chong Gyu Lee}
\address{Department of Mathematics, Soongsil University, Seoul 157642 Korea}
\email{cglee@ssu.ac.kr}

\author[J. H. Silverman]{Joseph H. Silverman}
\address{Department of Mathematics, Brown University, Providence, RI 02906 US.
  ORCID: https://orcid.org/0000-0003-3887-3248}
\email{jhs@math.brown.edu}

\date{\today}



\begin{abstract}
In this paper we study the locus of generalized degree $d$ H{\'e}non
maps in the parameter space $\operatorname{Rat_d^N}$ of degree $d$
rational maps $\mathbb{P}^N\to\mathbb{P}^N$ modulo the conjugation
action of $\operatorname{SL_{N+1}}$.  We show that H{\'e}non maps are
in the GIT unstable locus if $N\ge3$ or $d\ge3$, and that they are
semistable, but not stable, in the remaining case of $N=d=2$.  We also
give a general classification of all unstable maps in
$\operatorname{Rat}_2^2$.
\end{abstract}

\thanks{The first author is supported by Basic Science Research
  Program through the National Research Foundation of Korea(NRF)
  funded by the Ministry of Education. (NRF-2016R1D1A1B01009208)}

\maketitle

\section{Introduction}
When extending the theory of the dynamics of endomorphisms of~$\PP^2$
to more general rational maps, a standard collection of maps to study
are H{\'e}non  maps.  These maps were originally studied in dimension~$2$,
where they have the form
\begin{equation}
  \label{eqn:deffP}
  f_P(x,y) = \bigl(ay, bx + P(y) \bigr)\quad\text{with $a,b\in K^*$,
    $P(y)\in K[y]$, $\deg P=d\ge2$.}
\end{equation}
H{\'e}non  maps were the first maps shown to exhibit strange attractors in
their real dynamics. They have also received considerable attention in
the study of arithmetic and algebraic dynamics, since they share many
properties with endomorphisms, including the height-boundeness of the
set of preperiodic
points~\cite{denis:periodicaffineaut,marcello:CR2000}, the existence
of a canonical
height~\cite{ingramhenon2011,arxiv0405007,arxiv0909.3573,arxiv0909.3107,silverman:henonmap},
and equidistribution of periodic points~\cite{MR3127046}.

Various authors have considered generalizations of classical H{\'e}non  maps
to dimension greater than~$2$. For example, dynamics over~$\RR$ for
maps of the form
\begin{equation}
  \label{eqn:fPdefAN}
  f_P(x_1,\ldots,x_N)=\bigl(b_2x_2,b_3x_3,\ldots,b_Nx_N,b_1x_1+P(x_N)\bigr)
\end{equation}
have attracted attention; see for
exampe~\cite{baierklein1990,richter2002,sprott2006}. These papers
concentrate mainly on real dynamics and quadratic maps. Our aim in
this paper is to study stability properties, in the sense of geometric
invariant theory, of quite general H{\'e}non type maps, as given in the
following definition.

\begin{definition}
\label{definition:Henonx}
Let~$K$ be an algebraically closed field of characteristic~$0$,
let $N\ge{k}\ge{2}$ and $d\ge2$, and fix 
scalars~$b_1,\ldots,b_N\in{K^*}$ and polynomials
\[
   P_{i+1} \in K[x_k,\ldots,x_i],  \quad k\le i\le N,\quad \max_{k\le i\le N} \deg{P_{i+1}}=d.
\]
The associated \emph{generalized H{\'e}non  map} is the affine
automorphism $f_\bfP:\AA^N\to\AA^N$ defined by
\begin{equation}
  \label{eqn:fPaffine}
  f_\bfP := \Bigl(
    \overbrace{b_2x_2,\ldots,b_kx_k}^{\text{$k-1$ terms}},
    \overbrace{b_{k+1}x_{k+1}+P_{k+1}(\bfx),\ldots,b_Nx_N+P_N(\bfx)}^{\text{$N-k$ terms}},
  b_1x_1+P_{N+1}(\bfx) \Bigr).
\end{equation}
We denote the extension of~$f_\bfP$ to a birational map of~$\PP^N$ by
\[
  \fbar_\bfP : \PP^N \longrightarrow \PP^N.
\]
\end{definition}

The set of H{\'e}non  maps sits inside the parameter space~$\Rat_d^N$ of
degree~$d$ rational maps $\PP^N\to\PP^N$. In studying dynamics, one
looks at the quotient of this space by the natural conjugation action
of~$\SL_{N+1}$ on the maps parameterized by~$\Rat_d^N$. Geometric
invariant theory (GIT) provides subsets of stable and semistable
points in~$\Rat_d^n$ for which the $\SL_{N+1}$-quotient has a nice
structure.  For details of this standard construction, see for
example~\cite{MR2884382}. It is known~\cite{MR2741188,MR2567424} that
all endomorphisms of~$\PP^N$ are in the stable locus,
but the full structure of the stable and semistable loci is
complicated. It thus seems to be of interest to determine whether H{\'e}non 
maps, which are such a good testing ground for many dynamical
problems, are in these loci. Our main theorem provides an answer to
this question.

\begin{theorem}
\label{theorem:unstst}
Let~$N\ge{k}\ge2$ and $d\ge2$, and let
$\fbar_\bfP:\PP^N\to\PP^N$ be a degree~$d$ H{\'e}non  map defined over an
algebraically closed field of characteristic~$0$ as described in
Definition~$\ref{definition:Henonx}$.
\begin{itemize}
  \setlength{\itemsep}{0pt}
  \item[\textup{(a)}] If $d\ge3$ or $k\ge3$, then~$\fbar_\bfP$ is in the
    $\SL_{N+1}$-unstable locus of~$\Rat_d^N$.
  \item[\textup{(b)}] If $d=k=2$, then~$\fbar_\bfP$ is not in the
    $\SL_{N+1}$-stable locus of~$\Rat_d^N$.
  \item[\textup{(c)}] If $d=k=N=2$, then~$\fbar_\bfP$ is in the
    $\SL_{N+1}$-semistable locus of~$\Rat_d^N$.
\end{itemize}
\end{theorem}

\begin{remark}
We briefly discuss why one might be interested in the stability
properties of H{\'e}non  maps.  Let~$K$ be an algebraically closed field.
GIT says that the stable and semistable loci~$(\Rat_d^N)^\stab$
and~$(\Rat_d^N)^\semistab$ have $\SL_{N+1}$-quotients that are
irreducible algebraic varieties.  In particular, the stable
quotient~$\Moduli_d^N:=(\Rat_d^N)^\stab//\SL_{N+1}$ has the property
that two points~$f_1,f_2\in(\Rat_d^N)^\stab(K)$ have the same image
in~$\Moduli_d^N(K)$ if and only if there is some~$\phi\in\SL_{N+1}(K)$
such that~$f_2=\phi\circ{f_1}\circ\phi^{-1}$. For the semistable
quotient~$\overline\Moduli_d^N:=(\Rat_d^N)^\semistab//\SL_{N+1}$, two
points~$f_1,f_2\in(\Rat_d^N)^\semistab(K)$ have the same image
in~$\overline\Moduli_d^N(K)$ if and only if the Zariski closure of
their orbits~$\SL_{N+1}(K)f_1$ and~$\SL_{N+1}(K)f_2$ has a point in
common, which is somewhat less satisfactory, but is balanced by the
fact that~$\overline\Moduli_d^N$ is proper of~$K$. So for example, if
we are given a~$1$-parameter family of endomorphisms
\[
f : \AA^1\setminus0\longrightarrow \{ \text{degree $d$ morpshims $\PP^N\to\PP^N$} \},
\]
then the induced map to~$\overline\Moduli_d^N$ fills in with an
equivalence class of semistable maps at~$0$, i.e., there is a morphism
\[
  \AA^1(K)\longrightarrow\overline\Moduli_d^N(K) \quad
  \text{given for $t\ne0$ by}\quad t\longmapsto \langle f_t\rangle.
\]
It is thus of interest to understand the semistable maps
in~$\Rat_d^N$, since they are the maps that occur as the natural
limiting values in families of maps.
\par
For example, let $d\ge2$, and $t\in\AA^1(K)$, consider the family of maps
\begin{equation}
  \label{eqn:fbardtxyz}
  \fbar_{d,t}(x,y,z) = [tx^d+yz^{d-1},xz^{d-1}+y^d,z^d] : \PP^2\longrightarrow\PP^2.
\end{equation}
For $t\ne0$, the map~$\fbar_{d,t}$ is a morphism, so it is
in~$(\Rat_d^2)^\stab(K)$, and as~$t\to0$, the
points~$\langle{f_t}\rangle\in\overline\Moduli_d^2(K)$ have a limiting
value consisting of the~$\SL_3(K)$-orbits of various semistable maps
in~$(\Rat_d^2)^\semistab(K)$. Theorem~\ref{theorem:unstst} says that
if~$d=2$, then
\[
  \lim_{t\to0}\langle{f_t}\rangle = \langle f_0\rangle,
\]
but that if~$d\ge3$, then the H{\'e}non  map~$f_0$ is not a limiting value
of~$\langle{f_t}\rangle$. Roughly speaking, this says that for~$d=2$,
one can study the H{\'e}non  map~$f_0$ as a natural degeneration of the
family~\eqref{eqn:fbardtxyz} of morphisms~$f_t$, but that this is not
the case for~$d\ge3$.
\end{remark}

We conclude this introduction with a summary of the steps that go into
proving Theorem~\ref{theorem:unstst}.  We start in
Section~\ref{section:hmcrit} by setting notation and stating the
Hilbert-Mumford numerical criterion for
stability~\cite{mumford:geometricinvarianttheory}.  We use this
criterion in Section~\ref{section:instabdegge3} to show that H{\'e}non  maps
are unstable if $d\ge3$ or $k\ge3$, and that they are not stable if
$d=k=2$.  This leaves the problem of determining if H{\'e}non  maps are
semistable or unstable when $d=k=2$.  Since we do not know how to do this directly,
we follow a different path in the case $N=d=k=2$.  In
Section~\ref{section:classifyunstable} we show that all unstable maps
in $\Rat_2^2$ are either algebraically unstable or are linearly
fibered over~$\PP^1$, and we then show that this precludes their being
H{\'e}non  maps. Thus the proof of Theorem~\ref{theorem:unstst}(c) follows
the Shelock Holmes method: ``Once you eliminate the impossible,
whatever remains\dots must be the truth.''

\section{The Hilbert--Mumford numerical criterion for stability}
\label{section:hmcrit}

In this section, we review and set notation for the Hilbert--Mumford
numerical criterion on projective spaces.  For a $1$-parameter
subgroup ($1$-PS) $L:\GG_m\to\SL_{N+1}\to\Aut(\PP^N)$, we let~$L$ act
by conjugation on a rational map
\[
f= \left[ \sum a_1(\bfi) \bfX^\bfi: \sum a_2(\bfi) \bfX^\bfi: \cdots : \sum a_{N+1}(\bfi) \bfX^\bfi  \right]  
\in\Rat_d^N
\]
via
$f^L:=L(\alpha)\circ{f}\circ{L(\alpha)^{-1}}$. (Here~$\bfX:=[x_1:\cdots:x_{N+1}]$,
and we sum over multi-indices $\bfi:=(i_1,\ldots,i_{N+1})$.)  After 
a change of coordinates, we may take~$L$ to be diagonal, i.e., the
matrix of~$L(\alpha)$ is diagonal with
entries~$\alpha^{r_1},\ldots,\alpha^{r_{N+1}}$ for some integers~$r_i$
that sum to~$0$. Then the action takes the form
\[
 f^L = \left[ \sum \alpha^{e_1(\bfi,L)} a_1(\bfi) \bfX^\bfi:
   \sum \alpha^{e_2(\bfi,L)}a_2(\bfi) \bfX^\bfi: \cdots :
   \sum \alpha^{e_{N+1}(\bfi,L)}a_{N+1}(\bfi) \bfX^\bfi \right].
\]
The Hilbert--Mumford numerical criterion for stability is then defined
in terms of the following numerical invariant:\footnote{This is the negative of
  the~$\mu$ that typically appears in the literature, but we find that
  this version is easier to work with.}
\[
\mu(f, L)  := \min_{\substack{(j,\bfi)\\ a_j(\bfi)\ne 0\\}}   e_j(\bfi,L).
\]

\begin{theorem}[Numerical criterion, \cite{mumford:geometricinvarianttheory}]
  Let $f\in \Rat_d^N$. Then\footnote{More formally, $\Rat_d^N$ is the
    complement of a hypersurface in a large projective space~$\PP^M$,
    and stability is always relative to the ample line
    bundle~$\Ocal_{\PP^M}(1)$, so we omit it from the notation.}
\begin{itemize}
  \setlength{\itemsep}{0pt}
\item[\textup{(a)}]\qquad
  $\mu(f,L) > 0$ for some $1$-\textup{PS} $L$
  \quad $\Longleftrightarrow$ \quad
  $f$ is unstable.
\item[\textup{(b)}]\qquad
  $\mu(f,L) \ge 0$ for some $1$-\textup{PS} $L$
  \quad $\Longleftrightarrow$ \quad
  $f$ is not stable.
\end{itemize}
\end{theorem}

The Hilbert--Mumford criterion suggests if $f$ has a small number of
monomials, then $f$ has a higher chance of being unstable. Thus H{\'e}non 
maps, especially those of high degree or dimension, are likely to be
unstable.  Theorem~\ref{theorem:unstst} confirms this intuition.

\section{Instability of H{\'e}non  maps}
\label{section:instabdegge3}

In this section we use the Hilbert--Mumford criterion to prove the
first two parts of Theorem~\ref{theorem:unstst}, which says that no
H{\'e}non  map is stable and that most H{\'e}non  maps are unstable.

\begin{proof}[Proof of Theorem \textup{\ref{theorem:unstst}(a,b)}] 
We homogenize equation~\eqref{eqn:fPaffine}, using the new variable
by~$x_{N+1}$, to obtain a birational map
\[
\fbar_\bfP : \PP^N \longrightarrow \PP^N.
\]
Writing~$\Pbar_i(\bfx)$ for the homogenization of~$P_i(\bfx)$, the
birational map~$\fbar_\bfP:\PP^N\to\PP^N$ has the form
\begin{multline*}
\fbar_\bfP = \smash[b]{\Bigl[}
  \overbrace{b_2x_2x_{N+1}^{d-1} : \cdots : b_kx_kx_{N+1}^{d-1}}^{\text{coordinates $1,2,\ldots,k-1$}} :  \\
  \overbrace{b_{k+1}x_{k+1}x_{N+1}^{d-1}+\Pbar_{k+1}(\bfx) : \cdots : b_Nx_Nx_{N+1}^{d-1}+\Pbar_N(\bfx)}^{\text{coordinates $k,k+1,\ldots,N-1$}} : \\
  \underbrace{b_1x_1x_{N+1}^{d-1}+P_{N+1}(\bfx)}_{\text{coordinate $N$}} : 
  \underbrace{x_{N+1}^d}_{\text{coordinate $N+1$}} \Bigr].
\end{multline*}

We write~$I_n$ to denote the $n$-by-$n$ identity matrix.
For integers~$r,s,t$, not all~$0$ and satisfying
\begin{equation}
  \label{eqn:N1rst0}
  r(k-1)-s(N-k+1)-t = 0,\quad r,s,t \ge 0,
\end{equation}
we define a 1-PS $L:\GG_m\to\SL_{N+1}$ by the formula
\[
  L = L_{r,s,t} :=
  \begin{pmatrix}
    \alpha^r I_{k-1} & 0 & 0 \\
    0 &  \alpha^{-s} I_{N-k+1}   & 0 \\
    0 & 0 & \alpha^{-t} \\
  \end{pmatrix}.
\]
Explicitly,  the action of~$L_{r,s,t}$ on the monomials~$x_1,\ldots,x_{N+1}$ is given by
\[
  L_{r,s,t}(x_i)
  = \begin{cases}
    \alpha^{r}x_i &\text{for $1\le i\le k-1$,} \\
    \alpha^{-s}x_i &\text{for $k\le i\le N$,} \\
    \alpha^{-t}x_{N+1} &\text{for $k=N+1$.} \\
  \end{cases}
\]
We compute the powers of~$\alpha$ that appear in front of the
monomials in
\[
  \fbar_\bfP^L:=L_{r,s,t} \circ \fbar_\bfP \circ L_{r,s,t}^{-1}.
\]
The results ae given in Table~\ref{table:alphaexpfLx}.

In Table~\ref{table:alphaexpfLx}, only line~(IV), which deals with the
monomials appearing in the
polynomials~$\Pbar_{k+1},\ldots,\Pbar_{N+1}$, needs some further
explanation.  For each~$k\le{i}\le{N}$, the polynomial~$P_{i+1}$ is a
polynomial in the variables~$x_k,\ldots,x_i$ by definition, so in the
homogenized polynomials in~$\Pbar_{i+1}$, the monomials are of degree~$d$
in the variables~$x_k,\ldots,x_i,x_{N+1}$, i.e., they are monomials of
the form
\begin{equation}
  \label{eqn:monomialxkek}
  x_k^{e_k}\cdots x_i^{e_i} x_{N+1}^{d-e_k-\cdots-e_i}.
\end{equation}
The 1-PS~$L_{r,s,t}$ mutliplies each of~$x_k,\ldots,x_i$ by~$\alpha^{-s}$
and multiplies~$x_{N+1}$ by~$\alpha^{-t}$.  The monmomials of the
form~\eqref{eqn:monomialxkek} appear in the~$i$th coordinate
for some~$k\le{i}\le{N}$, so conjugation of~$\fbar_\bfP$ by~$L_{r,s,t}$
multiplies the monomial~\eqref{eqn:monomialxkek} by
\[
  \alpha^{(e_k+\cdots+e_i-1)s + (d-e_k-\cdots-e_i)t}.
\]
Letting~$m$ denote the quantity~$e_k+\cdots+e_i$, we see that every
monomial of every~$\Pbar_i$ in~$\fbar_\bfP$ is multiplied by
$\alpha^{(m-1)s+(d-m)t}$ for some $0\le{m}\le{d}$.

\begin{table}[ht]
\[
\begin{array}{|c|c|c|c|c|c|} \hline
   & \text{coordinate of $\fbar_\bfP$} & \text{monomial} & \text{exponent of $\alpha$} \\ \hline\hline
 \text{I} &   1 \le i \le k-2 & x_{i+1}x_{N+1}^{d-1} & (d-1)t  \\ \hline
 \text{II} &   k-1 & x_k x_{N+1}^{d-1}   & r+s+(d-1)t \\ \hline
 \text{III} &   k \le i \le N-1 & x_{i+1}x_{N+1}^{d-1} & (d-1)t \\ \hline
 \text{IV} &  k \le i\le N & \text{monomial in $\Pbar_{i+1}^{{\vphantom)}}$} & (m-1)s+(d-m)t~\text{for various $0\le m\le d$} \\ \hline
 \text{V} &   N & x_1x_{N+1}^{d-1} & -r -s + (d-1)t \\ \hline
 \text{VI} &   N+1 & x_{N+1}^d & (d-1)t \\ \hline
\end{array}
\]
\caption{Exponents of $\alpha$ used in computing $\mu(\fbar_\bfP,L_{r,s,t})$}
\label{table:alphaexpfLx}
\end{table}

\Case{1}{k\ge3~\textbf{or}~d\ge3}
By definition of H{\'e}non  maps, we always have~$N\ge{k}\ge2$ and~$d\ge2$. 
We set
\begin{equation}
  \label{eqn:sk1tk1N1r2N2k}
  s=k-1,\quad t=(k-1)(N+1),\quad r=2N+2-k,
\end{equation}
so~$(r,s,t)$ satisfy~\eqref{eqn:N1rst0} and are all strictly positive.
It follows immediately that the exponent of~$\alpha$ in
lines~I,~II,~III, and~VI of Table~\ref{table:alphaexpfLx} are strictly
positive.

For line~IV, we observe that if $1\le{m}\le{d}$, then $(m-1)s+(d-m)t$
is a sum of two non-negative terms, at least one of which is positive,
so the sum is positive. And for~$m=0$, using~\eqref{eqn:sk1tk1N1r2N2k}
shows that the exponent of~$\alpha$ is
\[
  dt-s = (k-1) \bigl( d(N+1)-1 \bigr) > 0.
\]

For line~V, we use~\eqref{eqn:sk1tk1N1r2N2k} and a little
algebra to compute
\[
  -r -s + (d-1)t= \bigl((d-1)(k-1)-2\bigr)(N+1) + 1 > 0.
\]
Note that this is where we need to assume that~$k\ge3$ or~$d\ge3$,
since if~$k=d=2$, then $(d-1)(k-1)-2=-1$, so $-r-s+(d-1)t=-N$.

This completes the proof that
\[
\text{$k\ge3$ or $d\ge3$}\quad\Longrightarrow\quad
\mu(\fbar_{\bfP},L_{r,s,t}) > 0,
\]
and hence by the numerical criterion that~$\fbar_{\bfP}$ is in the
$\SL_{N+1}$-unstable locus.

\Case{2}{k=2~\textbf{and}~d=2}
In this case we take
\[
r=1,\quad s=0,\quad t=1,
\]
which satisfy~\eqref{eqn:N1rst0}. The exponents of~$\alpha$ in the six
rows of Table~\ref{table:alphaexpfLx} are 
\[
  1,\; 0,\; 1,\; 2-m~\text{for $0\le m\le 2$},\; 0,\; 1.
\]
Hence
\[
  \mu(\fbar_{\bfP},L_{1,0,1}) = \min\{0,1,2\} = 0,
\]
which by the numerical criterion shows that~$\fbar_\bfP$ is not stable.
\end{proof}

\section{Unstable quadratic affine morphisms on $\AA^2$}
\label{section:classifyunstable}

In Section~\ref{section:instabdegge3} we showed that all H{\'e}non  maps are
not stable, but we were not able to show that certain quadratic maps
were actually unstable. There is a good reason for this, because at
least for~$N=2$, they are semistable, a fact that we prove in this
section.  However, since we do not see how to show directly that these
maps are semistable, we instead classify unstable quadratic maps
on~$\PP^2$, and then we show that this classification precludes an
unstable map from being a H{\'e}non  map.

\begin{theorem} 
\label{theorem:unsquadaff}
Let~$F:\PP^2\to\PP^2$ be a rational map having the following
properties\textup:
\begin{itemize}
  \setlength{\itemsep}{0pt}
\item
  $F:\PP^2\to\PP^2$ is a dominant rational map of degree $2$, i.e.,
  $F^*\Ocal_{\PP^2}(1)=\Ocal_{\PP^2}(2)$.
\item
  $F$ is  in the $\SL_3$-unstable locus of~$\Rat_2^2$.
\end{itemize}
Then one of the following is true\textup:
\begin{itemize}
  \setlength{\itemsep}{0pt}
  \item[\textup(a)]
The map~$F$ factors through a non-constant linear projection to~$\PP^1$,
i.e., there is a rational map~$\pi:\PP^2\to\PP^1$ satisfying
$\pi^*\Ocal_{\PP^1}(1)=\Ocal_{\PP^2}(1)$ and a rational
map~$G:\PP^1\to\PP^1$ making the following diagram
commute:
\begin{equation}
  \label{eqn:fpig}
    \begin{CD}
      \PP^2 @>F>> \PP^2 \\ @V\pi VV @VV\pi V \\ \PP^1 @>G>> \PP^1 \\
    \end{CD}
\end{equation}
\item[\textup(b)]
  The second iterate of~$F$ satisfies~$\deg(F^2)\le2$, so in
  particular, the map~$F$ is not algebraically stable in the sense of
  dynamical systems.\footnote{In general, if~$F:\PP^N\to\PP^N$ is a
    dominant rational map, then~$F$ is said to be \emph{algebraically
      stable}, in the sense of dynamics, if $\deg(F^n)=(\deg{F})^n$ for
    all~$n\ge1$.}
\end{itemize}
\end{theorem}
\begin{proof}
The assumption that~$F$ is unstable means that we can find a 1-PS
$L:\GG_m\to\SL_3$ satisfying~$\mu(F,L)>0$. We choose coordinates to
diagonalize~$L$ and so that the exponents of~$\alpha$ are
non-increasing.  This gives an~$L$ of the form
\[ 
  L = L_{r,s}(\alpha) 
  := \
  \begin{pmatrix}
    \alpha^{r} & 0 & 0\\
    0 & \alpha^{s-r} & 0\\
    0 & 0 & \alpha^{-s} \\
  \end{pmatrix}
  \quad\text{with $r\ge s-r \ge -s$ and $r,s>0$.}
\]

Table~\ref{table:expofalphax} lists the exponents of~$\alpha$ for the
monomials of $F^L:=L\circ{F}\circ{L}^{-1}$. Then~$\mu(F,L_{r,s})$ is
the smallest entry in the table whose monomial appears in~$F$.
Alternatively, our assumption that~$\mu(F,L_{r,s})>0$ means that every
non-positive entry in Table~\ref{table:expofalphax} forces the
corresponding monomial to not appear in~$F$. 

\begin{table}[ht]
\[
\begin{array}{|c||c|c|c|c|c|c|}
\hline
& x^2 & y^2 & z^2 & xy & yz & xz \\
\hline\hline
\text{$x$-coordinate of $F^L$}    &  -r & 3r-2s & r+2s & r-s & 2r & s \\ \hline       
\text{$y$-coordinate of $F^L$}    &  -3r+s & r-s & -r+3s & -r & s & -2r+2s \\ \hline  
\text{$z$-coordinate of $F^L$}    &  -2r-s & 2r-3s & s & -2s & r-s & -r \\ \hline     
\end{array}
\]
\caption{The exponent of $\alpha$ on each monomial of $F^L:=L \circ F \circ L^{-1}$}
\label{table:expofalphax}
\end{table}

\Case{1}{r\ge s}
In this case we look at the~$x^2$,~$xy$ and~$xz$ columns and the~$y$
and~$z$-coordinate rows in Table~\ref{table:expofalphax}. These six
entries are
\[
  \begin{array}{|c||c|c|c|c|c|c|} \hline
    & x^2 & xy & xz \\ \hline\hline
    \text{$y$-coordinate of $F^L$}   & -3r+s  &  -r  & -2r+2s \\ \hline
    \text{$z$-coordinate of $F^L$}   & -2r-s &  -2s & -r    \\ \hline
  \end{array}
\]
The fact that~$r,s>0$ immediately implies that four of the entries
are negative. Further, the entry~$-2r+2s$ is
non-positive because we are in the case that~$r\ge{s}$, and the entry~$-3r+s$
is negative because we normalized~$L$ so that~$r\ge{s-r}$.

It follows that the monomials~$x^2$,~$xy$, and~$xz$ do not appear in
the~$y$ and~$z$-coordinates of~$F$, so~$F$ has the form
\[
  F(x:y:z) = \bigl[ F_1(x:y:z): F_2(y:z): F_3(y:z) \bigr].
\]
Thus~$F$ factors over~$\PP^1$ as in~\eqref{eqn:fpig}
with
\[
\pi(x:y:z)=[y:z]\quad\text{and}\quad G(y:z)=\bigl[F_2(y:z):F_3(y:z)\bigr].
\]

\Case{2}{r < s}  
We rewrite Table~\ref{table:expofalphax}, putting an~$\xmark$ in the
boxes whose entries are non-positive. The results are compiled in
Table~\ref{table:expofalphaxrlts}.
\begin{table}[ht]
\[
\begin{array}{|c||c|c|c|c|c|c|}
\hline
& x^2 & y^2 & z^2 & xy & yz & xz \\
\hline\hline
\text{$x$-coordinate of $F^L$}    &  \xmark & 3r-2s & r+2s & \xmark & 2r & s \\ \hline       
\text{$y$-coordinate of $F^L$}    &  \xmark & \xmark & -r+3s & \xmark & s & -2r+2s \\ \hline  
\text{$z$-coordinate of $F^L$}    &  \xmark & \xmark & s & \xmark & \xmark & \xmark \\ \hline     
\end{array}
\]
\caption{The exponent of $\alpha$ on each monomial of $F^L$ assuming $2r\ge s>r>0$}
\label{table:expofalphaxrlts}
\end{table}

Examining Table~\ref{table:expofalphaxrlts}, we see that~$F$ has the form
\[
F(x:y:z) = \bigl[ \underbrace{\text{no $x^2$ or $xy$ term}}_{\text{$x$-coordinate}} :
  \underbrace{z\cdot(\text{linear term})}_{\text{$y$-coordinate}} :
  \underbrace{az^2}_{\text{$z$-coordinate}} \bigr].
\]
Hence~$[1,0,0]$ is in the indeterminacy locus of~$F$, and~$F$ sends
the entire line~$[u,v,0]$ to the indeterminacy point~$[1,0,0]$. It
follows that we have a strict
inequality~$\deg(F^2)<(\deg{F})^2$. Indeed, since the first coordinate
of~$F$ lacks both an~$x^2$ and an~$xy$ term, we see by a direct
calculation that~$\deg(F^2)\le2$.
\end{proof}

We now have the tools needed to complete the proof of
Theorem~\ref{theorem:unstst}. 

\begin{corollary}[Theorem \textup{\ref{theorem:unstst}(c)}]
\label{corollary:Henonsemistable}
Let $\fbar_P:\PP^2\to\PP^2$ be a H{\'e}non  map of degree~$2$, i.e.,
there are scalars~$a,b\in{K^*}$ and a polynomial~$P(y)\in{K[x]}$
of degree~$2$ such that~$\fbar_P$ is the 
extension to~$\PP^2$ of the affine automorphism
\begin{equation}
  \label{eqn:affinehenon}
  f_P:\AA^2\longrightarrow\AA^2,\quad
  f_P(x,y) = \bigl(ay,bx+P(y)\bigr).
\end{equation}
Then~$\fbar_P$ is $\SL_3$-semistable.
\end{corollary}
\begin{proof}
To ease notation in the proof of
Corollary~\ref{corollary:Henonsemistable}, we write~$f$ and~$\fbar$
for~$f_P$ and~$\fbar_P$.  We assume that~$\fbar$ is $\SL_3$-unstable
and derive a contradiction.  This assumption and
Theorem~\ref{theorem:unsquadaff} imply that either~$\deg(\fbar^2)\le2$
or~$\fbar$ is linearly fibered over~$\PP^1$.  A direct calculation
with~\eqref{eqn:affinehenon} yields~$\deg(f^2)=4$, so there is a
non-constant linear projection~$\pi:\PP^2\to\PP^1$ and a rational
map~$\gbar:\PP^1\to\PP^1$ satisfying~$\pi\circ{\fbar}=\gbar\circ\pi$,
as illustrated by the commutative diagram~\eqref{eqn:fpig}
with~$F=\fbar$ and~$G=\gbar$.

For each~$t\in\PP^1(K)$, we let
\[
  \ell_t := \overline{\pi^{-1}(t)} \subset \PP^2
\]
be the line lying over~$t$, i.e., the Zariski closure of the inverse
image of~$t$. Then the
semi-conjugation~$\pi\circ{\fbar}=\gbar\circ\pi$ implies
that\footnote{We use the usual convention that if~$F:X\to{Y}$ is a
  rational map between smooth projective varieties with indeterminacy
  locus~$I_f$, and if~$Z\subset{X}$ is a subvariety of
  codimension~$1$, then its image~$F(Z)$ is defined to
  be~$\overline{F(Z\smallsetminus{I_f})}$.}
\begin{equation}
  \label{eqn:f1}
  \fbar(\ell_t) \subseteq \ell_{\gbar(t)}
  \quad\text{for all $t\in\PP^1(K)$.}
\end{equation}
This suggests that we study the effect of~$\fbar$ on lines, as in the
next result.

\begin{lemma}
\label{claim:Henonline}
Let~$\ell\subset\PP^2$ be a line. Then its image under the H{\'e}non 
map~$\fbar$ is as follows:
\[ 
\fbar(\ell) = 
\begin{cases}
  \bigl\{[0:1:0]\bigr\} &\text{if $\ell=\{z=0\}$,} \\
  \{x=Az\}~\text{for some $A\in K$}&\text{if $\ell=\{y=Bz\}$ for some $B\in K$,} \\
  \text{an irreducible conic} &\text{otherwise.}\\    
\end{cases}
\]
\end{lemma}
\begin{proof}
It is clear that~$\fbar$ sends~$\{z=0\}$ to~$[0:1:0]$. For all other
lines, we work with the affine polynomial map~$f$ given
by~\eqref{eqn:affinehenon}.  The image of the horizontal
line~$\ell=\{y=B\}$ is
\[
  f\bigl( \{y=B\} \bigr) = \bigl\{ \bigl(aB,bt+P(B)\bigr) : t\in K \bigr\} = \{x=aB\}.
\]
Finally, every non-horizontal line has the form~$\{x=\lambda{y}+\mu\}$
for some~$\lambda,\mu\in{K}$. The image of this line is
\[
f\bigl( \{x=\lambda y+\mu\} \bigr) = \bigl\{ \bigl(at,b(\lambda t+\mu)+P(t)\bigr) : t\in K \bigr\} =
\bigl\{ y = b \lambda a^{-1} x + b \mu + P(a^{-1} x) \bigr\},
\]
which is an irreducible conic, since~$\deg P=2$ by assumption.  This
concludes the proof of Lemma~\ref{claim:Henonline}.
\end{proof}

The fact that~$\pi$ is a linear projection implies that the
lines~$\ell_t$ are distinct for distinct~$t$. In particular, there is
at most one~$t\in\PP^1(K)$ with~$\ell_{t}=\{z=0\}$. If there is such
a~$t$, we denote it by~$t_0$, and we define a set of points
\[
T = \{t_0\} \cup \bigl\{ t\in\PP^1(K) : \gbar(t)=t_0 \bigr\};
\]
and if~$\ell_t\ne\{z=0\}$ for all~$t$, then we set~$T=\emptyset$.  We
note that~$T$ is a finite set, since~$\gbar:\PP^1\to\PP^1$ is
non-constant, and that our definition of~$T$ ensures that for
all~$t\in\PP^1(K)\smallsetminus{T}$, neither of the lines~$\ell_t$
and~$\ell_{\gbar(t)}$ is equal to the line~$\{z=0\}$.

Let~$t\PP^1(K)\smallsetminus\{t_0\}$.  Then~$\ell_t\ne\{z=0\}$,
but~\eqref{eqn:f1} says that~$\fbar(\ell_t)$ is contained in the
line~$\ell_{\gbar(t)}$, so Lemma~\ref{claim:Henonline} implies
that~$\ell_t$ is a line of the form $\{y=Bz\}$ and that its
image~$\fbar(\ell_t)$ is a line of the form~$x=Az$. This proves:
\begin{equation}
  \label{eqn:f2}
\parbox[t]{.9\hsize}{\noindent For all $t\in\PP^1(K)\smallsetminus\{t_0\}$ there
  exist~$A_t,B_t\in{K}$ such that
  \[ \ell_t = \{y=B_tz\} \quad\text{and}\quad \fbar(\ell_t) = \ell_{\gbar(t)} = \{x=A_tz\}.\]
}
\end{equation}

Nest suppose that~$t\in\PP^1(K)\setminus{T}$. Then~$\gbar(t)\ne{t_0}$,
so we
may apply the first formula in~\eqref{eqn:f2} with~$t$ replaced by~$\gbar(t)$ to conclude
that
\begin{equation}
  \label{eqn:f3}
  \ell_{\gbar(t)} = \{y=B_{\gbar(t)} z\} \quad\text{for all $t\notin T$.}
\end{equation}
Combining~\eqref{eqn:f2}, and~\eqref{eqn:f3}, we find
that for all~$t\in\PP^1(K)\smallsetminus{T}$, we have
\[
  \{x=A_tz\} = \fbar(\ell_t) = \ell_{\gbar(t)} = \{y=B_{\gbar(t)} z\}.
\]
But a line of the form~$x=Az$ cannot equal a line of the form~$y=Bz$.
This contradication concludes the proof that~$\fbar$ is not unstable.
\end{proof}

\begin{remark}
The proof of Corollary~\ref{corollary:Henonsemistable}, which appears
to be somewhat ad hoc, comes down to showing that a H{\'e}non
map~$\fbar_P:\PP^2\to\PP^2$ is not linearly fibered over~$\PP^1$.  One
might be tempted to exploit the fact that~$\fbar_P$ is a birational map
with critical locus $\Crit(\fbar_P)=3\{z=0\}$, and that~$\fbar_P$ is the
extension of~$\PP^2$ of an affine automorphism $f_P:\AA^2\to\AA^2$.  So
it is instructive to keep in mind the affine automorphism
\[
  \phi:\AA^2\longrightarrow\AA^2,\quad
  \phi(x,y) = \bigl(x,y+P(x)\bigr),
\]
and its extension~$\bar\phi:\PP^2\to\PP^2$. The map~$\bar\phi$ is
unstable, satisfies $\Crit(\bar\phi)=3\{z=0\}$, and is linearly
fibered over~$\PP^1$ by the map $\pi(x:y:z)=[x:z]$.
\end{remark}


\end{document}